\documentclass[11pt]{amsart}

\usepackage{amssymb}
\usepackage[abbrev]{amsrefs}
\usepackage{hyperref}

\newcommand{\N}{\mathbb N}
\newcommand{\R}{\mathbb R}

\providecommand{\norm}[1]{\left\lVert#1\right\rVert}
\providecommand{\bignorm}[1]{\bigl\lVert#1\bigr\rVert}
\providecommand{\ip}[2]{\left\langle #1, #2 \right\rangle}
\newcommand{\A}{\operatorname{A}}
\newcommand{\F}{\operatorname{F}}

\allowdisplaybreaks
\numberwithin{equation}{section}

\theoremstyle{plain}
\newtheorem{theorem}{Theorem}[section]
\newtheorem{lemma}[theorem]{Lemma}
\newtheorem{corollary}[theorem]{Corollary}

\theoremstyle{remark}
\newtheorem{remark}[theorem]{Remark}

\title[Mean convergence theorems]{%
Mean convergence theorems with respect to attractive points
in a Hilbert space}

\author{Koji~Aoyama}
\address[K.~Aoyama]
{Graduate School of Social Sciences, Chiba University,
1-33, Yayoi-cho, Inage-ku, Chiba-shi, Chiba 263-8522, Japan}
\email{aoyama@bm.skr.jp}

\author{Masashi~Toyoda}
\address[M.~Toyoda]
{Department of Information Science, Toho University, 
Miyama, Funabashi, Chiba 274-8510, Japan}
\email{mss-toyoda@is.sci.toho-u.ac.jp}



\begin{document}

\begin{abstract}
 In this paper, we prove a mean convergence theorem for a mapping
 with an attractive point in a Hilbert space 
 by combining a quasinonexpansive extension of the mapping
 with a known mean convergence theorem for a quasinonexpansive
 mapping.
\end{abstract}

\maketitle

\section{Introduction}

The aim of this paper is to obtain a mean convergence theorem for a
mapping with an attractive point in the sense of Takahashi and
Takeuchi~\cite{MR2858695} in a Hilbert space.
Our main result is a generalization of a mean convergence result for
quasinonexpansive mappings \cite{MR2682871}*{Lemma 5.1}. 
However, we prove it by using \cite{MR2682871}*{Lemma 5.1} itself and a
quasinonexpansive extension of the given mapping
(Lemma~\ref{lem:extension}).

Takahashi and Takeuchi~\cite{MR2858695} introduced the notion of an
attractive point of a mapping defined on a Hilbert space, 
and they proved existence and mean convergence results for attractive
points of a certain nonlinear mapping \cite{MR2858695}*{Theorem 3.1}.
Moreover, they pointed out that
we can obtain existence and convergence theorems for a fixed point 
of the mapping from the results for an attractive point. 
Thus, according to Takahashi and Takeuchi~\cite{MR2858695}, 
the study of attractive points can be regarded as encompassing the
study of fixed points.
Similar results for attractive points are known; see, for example, 
\cites{MR2983907, 
MR3015118, 
MR3073499, 
MR3397117 
}. 

In this paper, we consider the converse direction.
That is to say, we derive a result for an attractive
point from a known
result for a fixed point. 
An important tool for achieving this is a lemma with respect to the
existence of a quasinonexpansive extension (Lemma~\ref{lem:extension}). 

This paper is organized as follows: 
In \S2, we recall some definitions and lemmas related to our results.
Then, in \S3, 
we prove a mean convergence theorem for a mapping with an attractive
point in a Hilbert space (Theorem~\ref{thm:mean})
by using a mean convergence theorem for a quasinonexpansive mapping
(Lemma~\ref{lem:quasi}). 
The result is a generalization of \cite{MR2970683}*{Lemma 4.2}. 
Finally, in \S4, from our main result
we deduce a mean convergence theorem for a $\lambda$-hybrid mapping 
in the sense of \cite{MR2682871} with an attractive point.

\section{Preliminaries}

Throughout the present paper, 
$H$ denotes a real Hilbert space, 
$\ip{\,\cdot\,}{\,\cdot\,}$ the inner product of $H$, 
$\norm{\,\cdot\,}$ the norm of $H$, 
$C$ a nonempty subset of $H$, 
$\R$ the set of real numbers, 
and $\N$ the set of positive integers. 
Strong convergence of a sequence $\{x_n\}$ in $H$ to $z\in H$ is denoted
by $x_n \to z$ and weak convergence by $x_n \rightharpoonup z$. 

Let $T\colon C\to H$ be a mapping. 
Then the set of fixed points of $T$ is denoted by $\F(T)$, that is, 
$\F(T) = \{z \in C\colon Tz=z\}$. 
A point $z \in H$ is said to be an \emph{attractive point} of
$T$ if $\norm{Tx - z} \leq \norm{x-z}$ for all $x \in
C$; see~\cite{MR2858695}. 
The set of attractive points of $T$ is denoted by $\A(T)$, that is, 
\[
\A(T) = \bigcap_{x\in C} \{ z \in H\colon \norm{Tx - z} \leq
\norm{x-z}\}. 
\]
It is clear from the definition that 
\begin{equation}\label{eqn:C_ATsubsetFT} 
C \cap \A(T) \subset \F(T). 
\end{equation}
Moreover, since 
$\norm{Tx - z} \leq \norm{x-z}$ is equivalent to 
$-2 \ip{Tx-x}{z} \leq \norm{x}^2 - \norm{Tx}^2$, 
it turns out that $\A(T)$ is closed and convex. 

Let $T\colon C\to H$ be a mapping and $F$ a nonempty subset of $H$. 
Then $T$ is said to be \emph{quasinonexpansive with respect to
$F$}~\cite{sqnIII} if $\norm{Tx - z} \leq \norm{x-z}$ for all $x\in C$
and $z\in F$; 
$T$ is said to be \emph{quasinonexpansive} if $\F(T) \ne \emptyset$
and $\norm{Tx - z} \leq \norm{x-z}$ for all $x\in C$ and $z\in \F(T)$; 
$T$ is said to be \emph{nonexpansive} if 
$\norm{Tx - Ty} \leq \norm{x-y}$ for all $x,y\in C$; 
$T$ is said to be \emph{$\lambda$-hybrid} if
\[ 
\norm{Tx-Ty}^2 \leq  \norm{x-y}^2 + 2 (1-\lambda) \ip{x-Tx}{y-Ty}
\]
or equivalently
\begin{align}
 \label{eqn:hybrid-equiv2}
 \begin{split}
  &\norm{Tx-Ty}^2 \\
  &\qquad
  \leq \norm{x-Ty}^2 + \norm{Ty-y}^2
  + 2 \ip{\lambda x + (1-\lambda)Tx - Ty}{Ty-y} 
 \end{split}
\end{align}
for all $x,y\in C$, where $\lambda \in \R$; 
see \cites{MR2682871, 
MR3017202, 
MR2981792 
} for more details. 
It is easy to verify the following:
\begin{itemize}
 \item if $\A(T) \ne \emptyset$, 
       then $T$ is quasinonexpansive with respect to~$\A(T)$;
 \item $T$ is quasinonexpansive
       if and only if $C \cap \A(T) \supset \F(T)
       \ne \emptyset$;
 \item if $T$ is a $\lambda$-hybrid mapping with a fixed point, 
       then $T$ is quasinonexpansive. 
\end{itemize}
Moreover, under the assumption that $C$ is closed and convex, we know
that if $T$ is quasinonexpansive, then $\F(T)$ 
is closed and convex; see~\cite{MR0298499}*{Theorem~1}.

Let $D$ be a closed convex subset of $H$. 
It is known that, for each $x \in H$, 
there exists a unique point $x_0 \in D$ such that 
\[
 \norm{x - x_0} = \min\{\norm{x-y}: y\in D\}. 
\]
Such a point $x_0$ is denoted by $P_D (x)$ and $P_D$ is called the
\emph{metric projection} of $H$ onto $D$.
It is known that the metric projection is nonexpansive; 
see~\cite{MR2548424} for more details. 

We recall the following results.

\begin{lemma}[\cite{MR1409591}*{Theorem 2.16 (iv)}, 
 \cite{MR2006529}*{Lemma 3.2}]\label{lem:TToyoda}
 Let $H$ be a Hilbert space,  $C$ a nonempty closed convex subset of
 $H$, and $\{x_n \}$ a sequence in $H$ such that 
 $\norm{x_{n+1} - y} \leq \norm{x_n -y}$ for all $y \in C$ and $n \in
 \N$. 
 Then $\{P_C (x_n)\}$ converges strongly to some point $z\in C$. 
\end{lemma}

\begin{lemma}[\cite{MR2682871}*{Lemma~5.1}]\label{lem:quasi}
 Let $H$ be a Hilbert space, $C$ a nonempty closed convex subset of $H$, 
 $T\colon C\to C$ a quasinonexpansive mapping, 
 $x\in C$, and $\{z_n\}$ a sequence in $C$ defined by
 $z_n = \frac{1}{n} \sum_{k=1}^n T^{k-1} x$
 for $n\in\N$, where $T^0$ is the identity mapping on $H$. 
 Then the sequence $\{P_{\F(T)} T^n x\}$ converges strongly, and
 moreover, if every weak cluster point of $\{z_n \}$ belongs to $\F(T)$, 
 then $\{z_n\}$ converges weakly to the strong limit of 
 $\{P_{\F(T)} T^n x\}$. 
\end{lemma}

\begin{remark}
 The conclusion that $\{z_n\}$ is a weakly convergent sequence 
 in Lemma \ref{lem:quasi} is derived from \cite{MR559375}*{Lemma 1}. 
\end{remark}

\begin{lemma}[\cite{MR2858695}*{Lemma 2.2}] \label{lem:TT}
 Let $H$ be a Hilbert space, $C$ a nonempty closed convex subset of
 $H$,
 $T\colon C\to C$ a mapping, and $z$ an attractive point of $T$. 
 Then $P_C (z) \in \F(T)$. 
\end{lemma}

Using Lemma~\ref{lem:TT}, we obtain the following: 

\begin{lemma} \label{lem:P_FT=P_AT}
 Let $H$ be a Hilbert space, $C$ a nonempty closed convex subset of
 $H$,
 and $T\colon C\to C$ a quasinonexpansive mapping. 
 Then $P_{\F(T)} (x) = P_{\A(T)} (x)$ for all $x \in C$. 
\end{lemma}

\begin{proof}
 Let $x \in C$ be fixed. 
 Since $P_{\A(T)} (x) \in \A(T)$, Lemma~\ref{lem:TT} implies that
 $P_C P_{\A(T)} (x) \in \F(T)$. Hence we have 
 \[
 \norm{P_{\F(T)} (x) -x} \leq \norm{P_C P_{\A(T)} (x) - x}. 
 \]
 Since $x =P_C (x)$ and $P_C$ is nonexpansive, it follows that 
 \[
 \norm{P_C P_{\A(T)} (x) - x} = \norm{P_C P_{\A(T)} (x) - P_C(x)}
 \leq \norm{P_{\A(T)} (x) -x}. 
 \]
 As a result, we have 
 $\norm{P_{\F(T)} (x) -x} \leq \norm{P_{\A(T)} (x) -x}$.
 Taking into account $C \cap \A(T) \supset \F(T)$, 
 we have $P_{\F(T)}(x) \in \A(T)$. 
 Therefore we conclude that $P_{\F(T)} (x) = P_{\A(T)} (x)$. 
\end{proof}

\begin{remark}
 Using Lemma~\ref{lem:P_FT=P_AT}, 
 we obtain characterizations of the limits of a sequence $\{S_nx\}$ 
 in \cite{MR2858695}*{Theorem 3.6} and \cite{MR3015118}*{Theorem 5.2}; 
 see also \cite{MR2761610}*{Theorem 4.1}. 
\end{remark}

\section{Mean convergence theorem}

In this section, using a mean convergence result for a quasinonexpansive
mapping (Lemma~\ref{lem:quasi}), 
we obtain a mean convergence theorem with respect to attractive points
(Theorem~\ref{thm:mean}). 

We first prove that a mapping with an attractive point has a
quasinonexpansive extension such that the set of fixed points 
of the extension is equal to that of attractive points
of the mapping. 

\begin{lemma}\label{lem:extension}
 Let $H$ be a Hilbert space, $C$ a nonempty subset of $H$,
 $T\colon C \to H$ a mapping with an attractive point, 
 and $\tilde{T} \colon H \to H$ a mapping defined by 
 \begin{equation}\label{e:extension}
 \tilde{T}x = 
 \begin{cases}
  Tx, & x\in [C \setminus \F(T)] \cup [\F(T)\cap \A(T)];\\
  P_{\A(T)}(x), & \text{otherwise.} 
 \end{cases}
\end{equation}
 Then $\A(T) = \F(\tilde{T})$
 and $\tilde{T}$ is quasinonexpansive (with respect to $\A(T)$). 
\end{lemma}

\begin{proof}
 Set $D =  [C \setminus \F(T)] \cup [\F(T)\cap \A(T)]$. 
 We first show $\A(T) \subset \F(\tilde{T})$. 
 Let $z \in \A(T)$. 
 Suppose that $z \in D$. 
 Taking into account~\eqref{eqn:C_ATsubsetFT}, 
 we know that $z \in \F(T)$. Thus we have $\tilde{T}z = Tz = z$, 
 and hence $z \in \F(\tilde{T})$. 
 On the other hand, suppose that 
 $z \notin D$. 
 Since $\F(P_{\A(T)}) = \A(T)$, 
 it follows that $\tilde{T}z = P_{\A(T)}(z) = z$, and hence
 $z \in \F(\tilde{T})$. Therefore, $\A(T) \subset \F(\tilde{T})$.

 We next show $\A(T) \supset \F(\tilde{T})$. 
 Let $z \in \F(\tilde{T})$. 
 Suppose that $z \in D$. 
 Then $z= \tilde{T} z = Tz$. Thus 
 $z\in \F(T)$. Since $z\in D$, we see that
 $z \in \F(T)\cap \A(T) \subset \A(T)$. 
 On the other hand, suppose that 
 $z \notin D$. 
 Since $\F(P_{\A(T)}) = \A(T)$, 
 it follows that $z = \tilde{T}z = P_{\A(T)}(z)$, and hence
 $z \in \A(T)$. Therefore, $\A(T) \supset \F(\tilde{T})$.

 We finally show that $\tilde{T}$ is quasinonexpansive with respect to
 $\A(T)$. 
 Let $x \in H$ and $z \in \A(T)$. 
 Suppose that $x \in D$. 
 Since $x \in C$ and $z \in \A(T)$, we have 
 $\bignorm{\tilde{T} x - z} = \norm{Tx -z} \leq \norm{x-z}$. 
 On the other hand, suppose that 
 $x \notin D$. 
 Since $\F(P_{\A(T)}) = \A(T)$ and $P_{\A(T)}$ is nonexpansive, we have
 \[
 \bignorm{\tilde{T} x - z}
 = \norm{P_{\A(T)} (x) - P_{\A(T)} (z)} \leq \norm{x -z}. 
 \]
 Therefore, $\tilde{T}$ is quasinonexpansive with respect to $\A(T)$. 
\end{proof}

Using Lemmas~\ref{lem:TToyoda}, \ref{lem:quasi},
and~\ref{lem:extension}, we obtain the following theorem: 

\begin{theorem}\label{thm:mean}
 Let $H$ be a Hilbert space, $C$ a nonempty subset of $H$, 
 $T\colon C\to C$ a mapping with an attractive point, 
 $x\in C$, and $\{z_n\}$ a sequence in $C$ defined by
 \begin{equation}\label{e:z_n-def2}
  z_n = \frac{1}{n} \sum_{k=1}^n T^{k-1} x      
 \end{equation}
 for $n\in\N$, where $T^0$ is the identity mapping on $C$. 
 Then the sequence $\{P_{\A(T)} T^n x\}$ converges strongly, and
 moreover, if every weak cluster point of $\{z_n\}$ belongs to $\A(T)$, 
 then $\{z_n\}$ converges weakly to the strong limit of 
 $\{P_{\A(T)} T^n x\}$. 
\end{theorem}

\begin{proof} 
 Let $u \in \A(T)$. Then we see that 
 \[
 \norm{T^{n+1} x - u} \leq \norm{T^n x - u}
 \]
 for all $n \in \N$. 
 Thus Lemma~\ref{lem:TToyoda} shows that 
 $\{P_{\A(T)} T^n x\}$ converges strongly. 

 We next prove the latter assertion.
 We assume that every weak cluster point of $\{z_n\}$ belongs to $\A(T)$
 and set $D= [C \setminus \F(T)] \cup [\F(T)\cap \A(T)]$. 
 We consider two cases: 
 (i) Suppose that $T^{m-1}x \in D$ for all $m \in \N$. 
 Let $\tilde{T}\colon H \to H$ be a mapping defined
 by~\eqref{e:extension}
 and $\{\tilde{z}_n\}$ a sequence in $H$ defined by 
 \[
 \tilde{z}_n = \frac{1}{n} \sum_{k=1}^n \tilde{T}^{k-1} x  
 \]
 for $n\in\N$, where $\tilde{T}^0$ is the identity mapping on $H$. 
 Then we see that $\tilde{T}^{n-1} x = T^{n-1}x$ for all $n \in \N$, 
 and moreover, Lemma~\ref{lem:extension} implies that $\A(T) =
 \F(\tilde{T})$ and
 $\tilde{T}$ is quasinonexpansive. 
 Thus $\tilde{z}_n = z_n$ and 
 $P_{\F(\tilde{T})} \tilde{T}^n x 
 = P_{\A(T)} \tilde{T}^n x = P_{\A(T)} T^n x$ for all $n \in \N$,
 and hence, by assumption, every weak cluster point of
 $\{\tilde{z}_n\}$ belongs to $\A(T) = \F(\tilde{T})$. 
 Therefore we deduce from Lemma~\ref{lem:quasi} that 
 $\{ \tilde{z}_n \}$ converges weakly to the strong limit of 
 $\{P_{\F(\tilde{T})} \tilde{T}^n x\}$. 
 This proves the desired conclusion.
 (ii) Suppose that there exists $m \in \N$ such that $T^{m-1}x \notin
 D$.
 Taking into account 
 \[
  C \setminus \left\{ [C \setminus \F(T)] \cup [\F(T)\cap \A(T)]
 \right\} = \F(T) \setminus \A(T),
 \]
 we know that $T^{m-1} x \in \F(T) \setminus \A(T)$. 
 Since $T^{m-1}x$ is a fixed point of $T$, it follows that 
 \[
 z_n = \frac{1}{n} \sum_{k=1}^n T^{k-1} x 
 = \frac{1}{n} \left(
 \sum_{k=1}^{m-1} T^{k-1}x + (n-m+1)T^{m-1}x \right)
 \]
 for all $n \in \N$ with $n \geq m$. 
 Hence $\lim_n z_n = T^{m-1}x \notin \A(T)$,
 which contradicts the assumption that 
 every weak cluster point of $\{z_n\}$ belongs to $\A(T)$. 
\end{proof}

\begin{remark}
 Taking into account Lemma~\ref{lem:P_FT=P_AT}, we know that
 Theorem~\ref{thm:mean} implies Lemma~\ref{lem:quasi}. 
\end{remark}

\begin{remark}
 Using Theorem~\ref{thm:mean}, 
 \cite{MR2970683}*{Theorem 3.3}, 
  and \eqref{eqn:C_ATsubsetFT}, 
 we obtain \cite{MR2970683}*{Lemma 4.2}. 
\end{remark}

\section{Application}

In this section, using Theorem~\ref{thm:mean}, 
we show a mean convergence theorem for a 
$\lambda$-hybrid mapping with an attractive point.

From the proof of \cite{MR2682871}*{Theorem 4.1}, 
we obtain the following fact. 
For the sake of completeness, we give the proof. 

\begin{lemma}\label{lem:fpt}
 Let $H$ be a Hilbert space, $C$ a nonempty subset of $H$, 
 $\lambda \in \R$, $T\colon C\to C$ a $\lambda$-hybrid mapping, 
 $x\in C$, and $\{z_n\}$ a sequence in $C$ defined by 
 \eqref{e:z_n-def2}
 for $n\in\N$, where $T^0$ is the identity mapping on $C$. 
 Suppose that $\{T^n x\}$ is bounded. 
 Then every weak cluster point of $\{z_n\}$ is an attractive point of $T$. 
\end{lemma}

\begin{proof}
 Let $y \in C$ be given. 
 Since $T$ is $\lambda$-hybrid, it follows
 from~\eqref{eqn:hybrid-equiv2} that
 \begin{multline*}
  0 \leq \norm{T^{k-1}x - Ty}^2 - \norm{T^k x - Ty}^2 + \norm{Ty-y}^2 \\
  + 2 \ip{\lambda T^{k-1} x + (1-\lambda)T^k x - Ty}{Ty-y}
 \end{multline*}
 for all $k\in \N$. 
 Summing these inequalities from $k=1$ to $n$ and dividing by $n$, 
 we have
 \begin{align}
  \begin{split}\label{eqn:n}
   0 &\leq \frac1n \left(
   \norm{x - Ty}^2 - \norm{T^{n} x - Ty}^2 \right) + \norm{Ty-y}^2 \\
   &\qquad + 2 \ip{\lambda z_n
   + (1-\lambda) \Bigl( \frac{n+1}{n} z_{n+1} - \frac{x}n \Bigr)
   - Ty}{Ty-y}.
  \end{split}
 \end{align}
 Let $z$ be a weak cluster point of $\{ z_n \}$. 
 Then there exists a subsequence $\{z_{n_i}\}$ of $\{z_n\}$ such that 
 $z_{n_i} \rightharpoonup z$. Thus it follows from~\eqref{eqn:n} that
\begin{align}
 \begin{split} \label{eqn:n_i}
  0 &\leq \frac1{n_i} \norm{x - Ty}^2 + \norm{Ty-y}^2 \\
  &\qquad + 2 \ip{\lambda z_{n_i} + (1-\lambda)
  \Bigl( \frac{n_i +1}{n_i} z_{n_i+1} - \frac{x}{n_i} \Bigr) - Ty}{Ty-y}
 \end{split}
\end{align}
 for all $i\in\N$. 
 Since
 \[
 z_{n_i +1} = \frac{n_i}{n_i +1} z_{n_i} + \frac1{n_i +1} T^{n_i}x
 \]
 and $\{T^{n_i}x\}$ is bounded, 
 $\{z_{n_i +1}\}$ also converges weakly to $z$. 
 Taking the limit $i\to\infty$ in~\eqref{eqn:n_i}, we have
 \[
 0 \leq \norm{Ty-y}^2 + 2 \ip{\lambda z + (1-\lambda) z  - Ty}{Ty-y}
 = \norm{y-z}^2 - \norm{Ty-z}^2. 
 \]
 This shows that $\norm{Ty-z} \leq \norm{y-z}$ for all $y \in C$, 
 and hence every weak cluster point of $\{z_n\}$ is an attractive point
 of $T$.
\end{proof}

Using Theorem~\ref{thm:mean} and Lemma~\ref{lem:fpt}, we obtain the
following corollary, which is a mean convergence theorem to an
attractive point of a $\lambda$-hybrid mapping. 

\begin{corollary}\label{c:l-hybrid_AT}
 Let $H$ be a Hilbert space, $C$ a nonempty subset of $H$, 
 $\lambda \in \R$, 
 $T\colon C\to C$ a $\lambda$-hybrid mapping with an attractive point, 
 $x\in C$, and $\{z_n\}$ a sequence in $C$ defined by~\eqref{e:z_n-def2}
 for $n\in\N$, where $T^0$ is the identity mapping on $C$. 
 Then $\{z_n\}$ converges weakly to the strong limit of 
 $\{P_{\A(T)}T^nx\}$. 
\end{corollary}

\begin{proof}
 Let $z \in \A(T)$. Then it turns out that 
 \[
 \norm{T^n x} \leq \norm{T^n x - z} + \norm{z}
 \leq \norm{x-z} + \norm{z}
 \]
 for all $n \in \N$. Thus $\{T^n x\}$ is bounded. 
 Since $T$ is $\lambda$-hybrid, it follows from Lemma~\ref{lem:fpt} that
 every weak cluster point of $\{z_n \}$ belongs to $\A(T)$. 
 Therefore Theorem~\ref{thm:mean} implies the conclusion.
\end{proof}

Using Corollary~\ref{c:l-hybrid_AT} and Lemma~\ref{lem:P_FT=P_AT}, we
obtain the following corollary, which is a mean convergence theorem to
a fixed point of a $\lambda$-hybrid mapping. 

\begin{corollary}[\cite{MR2682871}*{Theorem 5.2}]
 Let $H$ be a Hilbert space, $C$ a nonempty closed convex subset of
 $H$, $\lambda \in \R$, 
 $T\colon C\to C$ a $\lambda$-hybrid mapping with a fixed point, 
 $x\in C$, and $\{z_n\}$ a sequence in $C$ defined by~\eqref{e:z_n-def2}
 for $n\in\N$, where $T^0$ is the identity mapping on $C$. 
 Then $\{z_n\}$ converges weakly to the strong limit of 
 $\{P_{\F(T)}T^nx\}$. 
\end{corollary}

\begin{proof}
 We know that $T$ is quasinonexpansive, 
 and hence 
 $\emptyset \ne \F(T) \subset C \cap \A(T)$. 
 Therefore, $\A(T) \ne \emptyset$. 
 Lemma~\ref{lem:P_FT=P_AT} shows that $P_{\A(T)} T^nx = 
 P_{\F(T)}T^nx$ for all $n \in \N$. 
 Consequently, Corollary~\ref{c:l-hybrid_AT} implies the conclusion. 
\end{proof}

\section*{Acknowledgment}
This work was supported by the Research Institute for Mathematical
Sciences, a Joint Usage/Research Center located in Kyoto University.


\end{document}